\documentclass[a4paper,12pt]{amsart}
\usepackage[margin=1.3in]{geometry}

\usepackage[utf8]{inputenc}
\usepackage[T1]{fontenc}

\usepackage{amsthm, amssymb, amsmath, amsfonts, mathrsfs}

\usepackage{mathtools}
\usepackage[ocgcolorlinks, colorlinks=true, pdfstartview=FitV, linkcolor=blue, citecolor=blue, urlcolor=blue, pagebackref=false]{hyperref}

\usepackage{microtype}

\newtheorem{theorem}{Theorem}[section]

\newtheorem{lemma}[theorem]{Lemma}

\newtheorem{problem}[theorem]{Problem}

\newtheorem{remark}[theorem]{Remark}

\let\Section=\section
\def\section{\setcounter{equation}{0}\Section}
\newcommand{\R}{\mathbb{R}}

\def\RR{\mathbb{R}}

\def\EE{\mathbb{E}}

\def\eps{{\epsilon}}
\def\les{{\lesssim}}

\begin{document}

\title{Chaos expansion of 2D parabolic Anderson model}

\author{{ Yu \sc Gu}  \    {Jingyu \sc Huang}  
}

\address[Yu Gu]{Department of Mathematics, Carnegie Mellon University, Pittsburgh, PA 15213}

\address[Jingyu Huang]{Department of Mathematics, University of Utah, Salt Lake City, UT 84112}

%\thanks {\newline
%  Keywords: jijiji.     }
%\date{June 28, 2013}
\date{}

\begin{abstract}
We prove a chaos expansion for the 2D parabolic Anderson Model in small time, with the expansion coefficients expressed in terms of the density function of the annealed polymer in a white noise environment. \\

\end{abstract}

	\maketitle
	
\section{Introduction and main result}
%The aim of this note is to prove a Wiener chaos expansion of the solution to 
Consider the continuous parabolic Anderson model in $d=2$ formally written as 
\begin{equation}\label{eq:PAM}
\partial_tu(t,x)= \frac{1}{2}\Delta u(t,x) + u \cdot  (\dot{W}(x)-\infty), \quad t \geq 0\,, x \in \RR^2\,,
\end{equation}
where $\dot{W}(x)$ is a spatial white noise defined on the probability space $(\Omega,\mathcal{F},\mathbb{P})$ formally satisfying 
\[
\EE[\dot{W}(x)\dot{W}(y)]=\delta(x-y).
\]
% and we interpret the product $u\cdot \dot{W}(x)$ in the Stratonovich's sense. 
The equation \eqref{eq:PAM} was analyzed in \cite{gubinelli2015paracontrolled,hairer2014theory,hairer2015simple} by different approaches including the theory of regularity structures, para-controlled calculus, and the method of correctors and two-scale expansions. The main results in these references showed that a smoothed version of \eqref{eq:PAM} converges to some limit that is independent of the mollification. 
 
More precisely, let $\varphi:\R^2\to\R_+$ be a smooth and compactly supported function on $\RR^2$ satisfying $\varphi(x)=\varphi(-x)$ and $\int \varphi=1$. Define $\varphi_\eps(\cdot)=\eps^{-2}\varphi(\cdot/\eps)$ and
\begin{equation}
\dot{W}_{\epsilon}(x) = \int_{\RR^2} \varphi_{\epsilon}(x-y)dW(y)
\end{equation} as the mollification of $\dot{W}$. The covariance function of $\dot{W}_\eps$ is
\begin{equation}
R_\eps(x-y):=\EE [\dot{W}_{\epsilon}(x)\dot{W}_{\epsilon}(y)] = %\int_{\R^2}\varphi_\eps(x-y-z)\varphi_\eps(z)dz\,.
\varphi_{\epsilon}\star \varphi_{\epsilon}(x-y)\,.
\end{equation}
Let $u_\eps$ be the solution to the equation with smooth coefficients
\begin{equation}\label{e.pamsmooth}
\partial_t u_\eps(t,x)=\frac12\Delta u_\eps(t,x)+u_\eps(\dot{W}_\eps(x)-C_\eps),
\end{equation}
with the diverging constant
\begin{equation}\label{e.ceps}
C_\eps=\frac{1}{\pi}\log \eps^{-1}.
\end{equation}
Then $u_\eps$ converges in some weighted H\"older space to a limit $u$ that is defined to be the solution to \eqref{eq:PAM}, see \cite[Theorem 4.1]{hairer2015simple}. 

While the solution to \eqref{eq:PAM} is well-defined, its statistical property remains a challenge. We refer to \cite{allez2015continuous,cannizzaro2015multidimensional,chouk2016invariance,GX,MP} for some relevant discussions. The goal of this note is to provide a Wiener chaos expansion of the solution $u$, in the short time regime. We assume $u_\eps(0,x)=u_0(x)$ for some bounded function $u_0$. Theorem~\ref{t.mainth} below shows that for small $t$, $u_\eps(t,x)\to u(t,x)$ in $L^2(\Omega)$ as $\eps\to0$, and $u(t,x)$ is written explicitly as a Wiener chaos expansion in terms of the probability density of a polymer in a white noise environment, see \eqref{e.chaos}. We hope that the explicit chaos expansion will provide another way of proving the convergence to \eqref{eq:PAM}, e.g. from a discrete system using the general criteria proved in \cite{caravenna2013polynomial,MP}. The tool we use is a combination of the Feynman-Kac representation and Malliavin calculus. By writing $u_\eps(t,x)$ in terms of a chaos expansion, it suffices to pass to the limit in each chaos.

%To be more precise, let $\mathcal{D}(\RR^2)$ be the space of test functions on $\RR^2$, we consider a zero mean Gaussian family $W=\{W(\varphi); \varphi \in \mathcal{D}(\RR^2)\}$, defined in a complete probability space $(\Omega, \mathcal{F}, P)$, with covariance 
%\begin{equation}
%\EE [W(\varphi)W(\psi)] = \int_{\RR^2} \varphi(x) \psi(x) dx\,.
%\end{equation}
%The mapping $\varphi \to W(\varphi)$ defined in $\mathcal{D}(\RR^2)$ extends to a linear isometry between $L^2(\RR^2)$ and the Gaussian space spanned by $W$, denoted by 
%\begin{equation}
%W(\phi) = \int_{\RR^2} \phi(x) W(dx)
%\end{equation}
%for $\phi \in L^2(\RR^2)$. If $\phi$ and $\psi$ are in $L^2(\RR^2)$, then $\EE[W(\phi)W(\psi)] = \langle \phi, \psi\rangle_{L^2(\RR^2)}$. 

\subsection{Elements of Malliavin calculus}

We give a brief introduction to Malliavin calculus and refer to \cite{Nualart} for more details. For any function $\phi\in L^2(\R^2)$, we define $W(\phi)=\int \phi \ dW$. Let $F$ be a smooth and cylindrical random variable of the form
\begin{equation*}
F= f(W(\phi_1), \dots, W(\phi_n))\,,
\end{equation*}
with $\phi_i \in L^2(\RR^2)$, $f \in C_p^{\infty}(\RR^n)$ (namely $f$ and all its partial derivatives have polynomial growth), then the Malliavin derivative of $F$, denoted by $DF$, is the $L^2(\RR^2)-$valued random variable defined by 
\begin{equation}
D F = \sum_{j=1}^n \frac{\partial f}{\partial x_j} (W(\phi_1), \dots, W(\phi_n)) \phi_j\,.
\end{equation}
For each positive integer $k$, $D^kF$ is defined to be the $k$-th iterated derivative of $F$, which is a random variable taking values in $L^2(\RR^2)^{\otimes k}$, the $k$-th tensor product of $L^2(\RR^2)$. %( actually it is even in $L^2(\RR^2)^{\odot}$ the $k$-th symmetric tensor product of $L^2(\RR^2)$).
The operator $D^k$ is closable from $L^2(\Omega)$ into $L^2(\Omega; L^2(\RR^2)^{\otimes k})$ and we define the Sobolev space $\mathbb{D}^{k, 2}$ as the closure of the space of smooth and cylindrical random variables under the norm 
\begin{equation}
\|D^k F\|_{k,2} = \sqrt{\EE \left[F^2 + \sum_{j=1}^k \|D^j F\| ^2_{L^2(\RR^2)^{\otimes j}}\right]} \,.
\end{equation}
Define $\mathbb{D}^{\infty,2} = \bigcap_{k=1}^{\infty} \mathbb{D}^{k, 2}$, and $L^2(\RR^2)^{\odot k}$ as the $k$-th symmetric tensor product of $L^2(\RR^2)$.

%We next give a small account on the notion of chaos expansions. 
For any integer $n\geq 0$, we denote by ${ \bf H}_n$ the $n$-th Wiener chaos of $W$. We recall that ${\bf H}_0$ is simply $\RR$, and for $n\geq 1$, ${\bf H}_n$ is the closed linear subspace of $L^2(\Omega)$ generated by the random variables 
\[
\{H_n(W(h)): h \in L^2(\RR^2), \|h\|_{L^2(\RR^2)}=1\},
\]
 where $H_n$ is the $n$-th order Hermite polynomials. For any $n\geq 1$, the mapping 
 \[
 I_n(h^{\otimes n}) := H_n(W(h))
 \]
  can be extended to a linear isometry between $L^2(\RR^2)^{\odot n}$ and ${\bf H}_n$, with the isometric relation 
\begin{equation}
\EE [I_n(h^{\otimes n})^2] = n! \|h^{\otimes n}\|^2_{L^2(\RR^2)^{\otimes n}}\,.
\end{equation}
Consider now a random variable $F\in L^2(\Omega)$, %measurable with respect to the $\sigma$-field $\mathcal{F}$ generated by $W$. This random variable 
it can be written as
\begin{equation}
F = \EE [F] + \sum_{n=1}^{\infty} I_n(f_n)\,,
\end{equation}
where the series converges in $L^2(\Omega)$, and the coefficients $f_n \in L^2(\RR^2)^{\odot n}$ are determined by $F$. This identity is called the Wiener-chaos expansion of $F$. 

When the above $F \in \mathbb{D}^{\infty, 2}$, the $n-$th coefficient $f_n$ in the Wiener chaos expansion of $F$ can be explicitly written as \cite[Page 3, equation (7)]{Stroock}
\begin{equation}\label{eq:Stroock formula}
f_n = \frac{\EE [D^n F]}{n!}\,.
\end{equation}

\subsection{Brownian self-intersection local time and polymer in white noise}

The self-intersection local time of the planar Brownian motion is a classical subject in probability theory \cite{chen2010random,le1992some,le1994exponential,varadhan1969appendix,yor1986precisions}. In the following, we discuss its connections to the parabolic Anderson model.

Using the Feynman-Kac formula, we write the solution to \eqref{e.pamsmooth} as 
\begin{equation}\label{e.fkre}
u_{\epsilon}(t,x) = \EE_{B} \left[ u_0(x+B_t) \exp \left( \int_0^t \dot{W}_{\epsilon}(x+B_s) ds - C_\eps t\right)\right]\,,
\end{equation}
where $B$ is a standard Brownian motion starting from the origin which is independent from $\dot{W}$, and $\EE_B$ denotes the expectation with respect to $B$. Taking expectation with respect to $\dot{W}_\eps$ and using the fact that the exponent inside the expectation in \eqref{e.fkre} is of Gaussian distribution for each realization of the Brownian motion, we obtain 
\[
\EE[u_\eps(t,x)]=\EE_B\left[u_0(x+B_t) \exp\left(\int_0^t\int_0^s R_\eps(B_s-B_u)duds-C_\eps t\right)\right],
\]
where we recall that $R_\eps$ is the covariance function of $\dot{W}_\eps$. It is well-known that
\begin{equation}\label{e.nov162}
\gamma_\eps(t,B):=\int_0^t\int_0^s R_\eps(B_s-B_u)duds-\int_0^t\int_0^s \EE_B[R_\eps(B_s-B_u)]duds\to \gamma(t,B)
\end{equation}
almost surely, and $\gamma(t,B)$ is the so-called renormalized self-intersection local time of the planar Brownian motion formally written as 
\begin{equation}\label{e.polymer}
\gamma(t,B)=\int_0^t\int_0^s \delta(B_s-B_u)duds-\int_0^t\int_0^s \EE_B[\delta(B_s-B_u)]duds.
\end{equation}
In addition, there exists some critical $t_c>0$ such that 
\begin{equation}\label{e.exponentialintegrability}
\EE_B[\exp(\gamma(t,B))]\left\{\begin{array}{cc}
<\infty & t<t_c,\\
=\infty & t>t_c.
\end{array}
\right.
\end{equation}
The renormalization constant in \eqref{e.ceps} matches the expectation in \eqref{e.nov162} up to an $O(1)$ correction, and a calculation as in \cite[Lemma 1.1]{GX} shows that there exists constants $\mu_1,\mu_2$ such that 
\begin{equation}\label{e.remainder}
\int_0^t \int_0^s \EE_B[R_\eps(B_s-B_u)]duds-C_\eps t\to t(\mu_1+\mu_2\log t)
\end{equation}
as $\eps\to0$. For small $t$, it was shown in \cite{GX} that
\begin{equation}
\EE[u(t,x)]=\lim_{\eps\to0}\EE[u_\eps(t,x)]=e^{t(\mu_1+\mu_2\log t)}\EE_B[u_0(x+B_t) e^{\gamma(t,B)}].
\end{equation}
This motivates us to define
\[
F(t):=\log \EE_B[e^{\gamma(t,B)}],
\]
so we can write
\[
\EE[u(t,x)]=e^{t(\mu_1+\mu_2\log t)+F(t)}\hat{\EE}_{t,B}[u_0(x+B_t)],
\]
where $\hat{\EE}_{t,B}$ denotes the expectation with respect to the Wiener measure tilted by the factor $e^{\gamma(t,B)}$, i.e., 
\[
\hat{\EE}_{t,B}[X]= \frac{\EE_B[Xe^{\gamma(t,B)}]}{\EE_B[e^{\gamma(t,B)}]}=\EE_B[X e^{\gamma(t,B)}] e^{-F(t)}
\]
for any bounded $X$. By the formal expression in \eqref{e.polymer}, we can view $\hat{\EE}_{t,B}$ as the expectation with respect to the annealed measure of a polymer in a white noise environment. By \eqref{e.exponentialintegrability}, it is clear that the measure is absolutely continuous with respect to the Wiener measure for small $t$. Applying the Radon-Nikodym theorem, for any $n\in \mathbb{Z}_+$ and $0<s_1<\ldots<s_n\leq t<t_c$, there exists a non-negative measurable function, denoted by 
\[
\mathscr{F}_{s_1,\ldots,s_n}: \R^{2n}\to\R,
\] such that 
\[
\hat{\EE}_{t,B}[1_{A}(B_{s_1},\ldots,B_{s_n})]=\int_A \mathscr{F}_{s_1,\ldots,s_n}(x_1,\ldots,x_n)dx
\]
for all $A\subset \R^{2n}$. In other words, $\mathscr{F}_{s_1,\ldots,s_n}$ is the joint spatial density function of the polymer path at $s_1<\ldots<s_n$. We note that $\mathscr{F}$ actually depends on $t$ since the tilted measure depends on $t$. For our purpose, we use the simplified notation since $t$ is fixed. It is an elementary exercise to show that $\mathscr{F}_{s_1,\ldots,s_n}(x_1,\ldots,x_n)$ is jointly measurable in $(s_1,\ldots,s_n,x_1,\ldots,x_n)$. For the convenience of the reader, we present a proof in the appendix.

Denote $[0,t]_<^n:=\{0\leq s_1<\ldots<s_n\leq t\}$, the following is our main result. 
\begin{theorem}\label{t.mainth}
There exists $t_0>0$ such that for each $t\in (0,t_0),x\in \RR^2$, the random variable $u_\eps(t,x)$ converges in $L^2(\Omega)$ to %denoted by 
\begin{equation}\label{e.chaos}
u(t,x)= \sum_{n=0}^{\infty} I_n(f_n(\cdot; t,x))\,.
\end{equation}
The coefficient $f_n(\cdot; t,x)$ is given by% the limit as $\epsilon \to 0$ of $f_{\epsilon, n}(\cdot; t,x)$ in $L^2(\RR^2)^{\otimes n}$ given by
\begin{equation}\label{e.excoeff}
\begin{aligned}
&f_n(y_1,\ldots,y_n; t,x) \\
&= e^{t(\mu_1+\mu_2\log t)+F(t)} \int_{\R^2} \int_{[0,t]_<^n}u_0(x+z)\mathscr{F}_{s_1,\ldots,s_n,t}(y_1-x,\ldots,y_n-x,z)dsdz.
\end{aligned}
\end{equation}
\end{theorem}

\begin{remark}
The small time constraint in Theorem~\ref{t.mainth} seems necessary. It was shown in \cite{GX} that $\EE[u(t,x)^2]$ is finite and admits a Feynman-Kac representation for small $t$, and we expect that $\EE[u(t,x)^2]=\infty$ when $t$ is large, in light of \eqref{e.exponentialintegrability}.
\end{remark}

\begin{remark}
Since the formal product $u\cdot\dot{W}$ in \eqref{eq:PAM} comes from the classical physical product $u_\eps\dot{W}_\eps$ in \eqref{e.pamsmooth}, we may interpret it in the Stratonovich's sense. If it is replaced by the Wick product: 
\begin{equation}\label{e.pamwick}
\partial_tu(t,x)=\frac12\Delta u(t,x)+u(t,x)\diamond \dot{W}(x),
\end{equation}
a different chaos expansion was proved in \cite{hu2002chaos}. Compared with \eqref{e.chaos}, the only difference is the lack of the weight $e^{\gamma(t,B)}$ in the definition of $\mathscr{F}$. This reduces the polymer measure to the original Wiener measure, in which case we have
\begin{equation}
\mathscr{F}_{s_1,\ldots,s_n}(x_1,\ldots,x_n)=q_{s_1}(x_1)q_{s_2-s_1}(x_2-x_1)\ldots q_{s_n-s_{n-1}}(x_n-x_{n-1}),
\end{equation}
where $q_t(x):=(2\pi t)^{-1}e^{-|x|^2/2t}$ is the standard heat kernel. With $\mu_1=\mu_2=F=0$, the expansion coefficient is given by
\begin{equation}
\begin{aligned}
&f_n(y_1,\ldots,y_n; t,x)\\
&=\int_{\R^2} \int_{[0,t]_<^n}u_0(x+z)\mathscr{F}_{s_1,\ldots,s_n,t}(y_1-x,\ldots,y_n-x,z)dsdz\\
&=\int_{[0,t]^n}\int_{\R^2}1_{\{0<s_n<\ldots<s_1<t\}}\prod_{j=0}^n q_{s_j-s_{j+1}}(y_j-y_{j+1})u_0(z)dzds,%}q_{t-s_1}(x-y_1)q_{s_1-s_2}(y_1-y_2)\ldots q_{s_{n-1}-s_n}(y_{n-1}-y_n)q_{s_n}(y_n-z)u_0(z)dzds.
\end{aligned}
\end{equation}
with the convention that $y_0=x,y_{n+1}=z,s_{n+1}=0$. Thus, the resulting chaos expansion is obtained by iterating the mild formulation of \eqref{e.pamwick}. %Theorem~\ref{t.mainth} shows that it leads to the right answer only if we tilt the Wiener measure by the exponential self-intersection local time.
The missing exponential weight $e^{\gamma(t,B)}$ favors self-attracting of the polymer paths, which prevails in the intermittency behaviors of parabolic Anderson model. We refer to the recent monograph \cite{konig} for more details.
\end{remark}

\begin{remark}
The same proof works in the one dimensional case, where the small time constraint can be removed, and there is no need to renormalize. A similar expansion coefficient as \eqref{e.excoeff} holds.
\end{remark}

\section{Proof of the main result}

%Let $\varphi(x)$ be a smooth nonnegative function on $\RR^2$ such that $\int_{\RR^2}\varphi(x)=1$. Let  $\dot{W}_{\epsilon}(x) = \int_{\RR^2} \varphi_{\epsilon}(x-y)W(dy)$, the regularization of $\dot{W}$. We note that 
%\begin{equation}
%\EE [\dot{W}_{\epsilon}(x)\dot{W}_{\epsilon}(y)] = \varphi_{\epsilon}* \varphi_{\epsilon}(x-y)\,.
%\end{equation}
%We follow the procedure of \cite{HL} to give a meaning to the solution to equation \eqref{eq:PAM}. 
%Let $u_{\epsilon}$ be the solution to the SPDE
%\begin{equation}
%\frac{\partial }{\partial t} u_{\epsilon} = \frac{1}{2}\Delta u_{\epsilon} + u_{\epsilon}(\dot{W}_{\epsilon} - C_{\epsilon})\,,
%\end{equation}
%here $C_{\epsilon} \sim \frac{1}{\pi} \log \frac{1}{\epsilon}$ as $\epsilon \to 0$. Actually, the $C_{\epsilon}$ can be regarded as 
%\begin{equation}
%C_{\epsilon} = \frac{1}{2} \EE \int_0^t \int_0^t \varphi_{\epsilon}* \varphi_{\epsilon}(B_s-B_r) ds dr\,.
%\end{equation}
%So, an application of the classical Feynman-Kac formula gives the formula for the solution $u_{\epsilon}(t,x)$: 
%\begin{align*}
%u_{\epsilon}(t,x) = &\EE_{B} \left( u_0(x+B_t) \exp \left( \int_0^t \dot{W}_{\epsilon}(x+B_s) ds - \frac{1}{2} \EE \int_0^t \int_0^t \varphi_{\epsilon}* \varphi_{\epsilon}(B_s-B_r) ds dr \right)\right)\,.
%\end{align*}
%The main result in \cite{HL} is that for each $t >0$, $u_{\epsilon}(t,x)$ converges in probability to a random variable $u(t,x)$ which is defined to the be the solution to equation \eqref{eq:PAM}. 

 For fixed $t>0,x\in\R^2,\epsilon>0$ and each realization of the Brownian motion, we write the exponent in \eqref{e.fkre} as
\[
\begin{aligned}
\int_0^t \dot{W}_\eps(x+B_s)ds&=\int_0^t \int_{\R^2} \varphi_\eps(x+B_s-y)dW(y) ds\\
&=\int_{\R^2} \left(\int_0^t \varphi_\eps(x+B_s-y)ds\right)dW(y)\\
&=\int_{\R^2} \Phi_{t,x,B}^\eps(y)dW(y),
\end{aligned}
\]
 with 
\[
\Phi_{t,x,B}^\eps(y):= \int_0^t \varphi_\eps(x+B_s-y)ds.
\]
Then it is easy to see that $u_{\epsilon}(t,x) \in \mathbb{D}^{\infty,2}$, and
\begin{equation}
\begin{aligned}
D^nu_\eps(t,x)=&\EE_B\left[ u_0(x+B_t) D^n \exp\left(\int_{\R^2}\Phi^\eps_{t,x,B}(y)dW(y)-C_\eps t\right)\right]\\
=&\EE_B\left[u_0(x+B_t) \exp\left(\int_{\R^2}\Phi^\eps_{t,x,B}(y)dW(y)-C_\eps t\right) (\Phi^\eps_{t,x,B}(\cdot))^{\otimes n}\right].
\end{aligned}
\end{equation}
By the Stroock's formula \eqref{eq:Stroock formula}, we can write the Wiener chaos expansion of $u_\eps(t,x)$ as 
\begin{equation}\label{eq:chaos u epsilon}
u_{\epsilon}(t,x)= \sum_{n=0}^\infty I_n\left( f_{\epsilon, n}(\cdot; t,x)\right)\,,
\end{equation}
with 
\begin{equation}
\begin{aligned}
f_{\epsilon, n}(\cdot; t,x) = &\frac{1}{n!} \EE[D^n u_\eps(t,x)]\\%=\frac{1}{n!}\EE\EE_B[ u_0(x+B_t) D^n \exp(\int_{\R^2}\Phi^\eps_{t,x,B}(y)dW(y)-C_\eps t)] \\
=&\frac{1}{n!} \EE_B \left[u_0(x + B_t) \exp\left(\int_0^t\int_0^s R_\eps(B_s-B_u)duds-C_\eps t\right)  \left(\Phi^\eps_{t,x,B}(\cdot)\right) ^{\otimes n}\right]\,.
 \end{aligned}
 \end{equation}
 By \eqref{e.remainder}, we define 
\begin{equation}
\begin{aligned}
r_\eps:=\int_0^t \int_0^s \EE_B[R_\eps(B_s-B_u)]duds-C_\eps t
 -t(\mu_1+\mu_2\log t),
\end{aligned}
\end{equation}
which goes to zero as $\eps\to0$, and rewrite
\begin{equation}
f_{\eps,n}(\cdot; t,x)=\frac{e^{t(\mu_1+\mu_2\log t)+r_\eps}}{n!}\EE_B\left[u_0(x+B_t) \exp(\gamma_\eps(t,B))  \left(\Phi^\eps_{t,x,B}(\cdot)\right) ^{\otimes n}\right].
\end{equation}

To prove Theorem~\ref{t.mainth}, it suffices to show that as $\eps\to0$,
\begin{equation}\label{e.conl2}
\sum_{n=0}^\infty n!\|  f_{\epsilon, n}(\cdot; t,x)-f_n(\cdot; t,x)\|^2_{L^2(\RR^2)^{\otimes n}}\to0.
\end{equation}
 Define 
\begin{equation}
\tilde{f}_{\eps,n}(\cdot; t,x):=\frac{e^{t(\mu_1+\mu_2\log t)+r_\eps}}{n!}\EE_B\left[u_0(x+B_t) \exp(\gamma(t,B))  \left(\Phi^\eps_{t,x,B}(\cdot)\right) ^{\otimes n}\right].
\end{equation}
Since 
\[
\begin{aligned}
&\|  f_{\epsilon, n}(\cdot; t,x)-f_n(\cdot; t,x)\|^2_{L^2(\RR^2)^{\otimes n}} \\
&\leq 2\|  f_{\epsilon, n}(\cdot; t,x)-\tilde{f}_{\eps,n}(\cdot; t,x)\|^2_{L^2(\RR^2)^{\otimes n}}+2\|  \tilde{f}_{\epsilon, n}(\cdot; t,x)-f_n(\cdot; t,x)\|^2_{L^2(\RR^2)^{\otimes n}}\,,
\end{aligned}
\]
the proof of \eqref{e.conl2} reduces to the following three lemmas.

\begin{lemma}\label{l.bound}
There exists $t_0,C>0$ independent of $\eps,n$ such that if $t<t_0$, 
\[
\|f_{\epsilon, n}(\cdot; t,x)\|_{L^2(\RR^2)^{\otimes n}}^2+\|\tilde{f}_{\eps,n}(\cdot; t,x)\|^2_{L^2(\RR^2)^{\otimes n}} +\|f_{n}(\cdot; t,x)\|^2_{L^2(\RR^2)^{\otimes n}}\leq \frac{(Ct)^n}{n!}.
\]
\end{lemma}

\begin{lemma}\label{l.con1}
There exists $t_0>0$ such that if $t<t_0$, 
\[
\|f_{\epsilon, n}(\cdot; t,x)-\tilde{f}_{\eps,n}(\cdot; t,x)\|^2_{L^2(\RR^2)^{\otimes n}}\to 0, \   \ \mbox{ as }  \eps\to0.
\]
\end{lemma}

\begin{lemma}\label{l.con2}
There exists $t_0>0$ such that if $t<t_0$, 
\[
\|\tilde{f}_{\eps,n}(\cdot; t,x)-f_n(\cdot; t,x)\|^2_{L^2(\RR^2)^{\otimes n}}\to 0, \   \  \mbox{ as } \eps\to0.
\]
\end{lemma}

In the following, we use the notation $a\les b$ when $a\leq Cb$ for some constant $C>0$ independent of $\eps,n$.

\noindent\emph{Proof of Lemma~\ref{l.bound}}. The proof of $f_{\eps,n}$ and $\tilde{f}_{\eps,n}$ is the same. Take $f_{\eps,n}$ for example:
\[
\begin{aligned}
&\|f_{\epsilon, n}(\cdot; t,x)\|_{L^2(\RR^2)^{\otimes n}}^2\\
&\les \frac{1}{(n!)^2} \int_{\RR^{2n}}\EE_{B^1,B^2} \left[\prod_{j=1}^2 \left(e^{\gamma_\eps(t,B^j)} \prod_{k=1}^n \Phi^\eps_{t,x,B^j}(y_k)\right)\right] dy,
\end{aligned}
\]
where $B^1,B^2$ stand for independent Brownian motions. Performing the integral in the $y$ variable, the r.h.s. of the above display is bounded by 
\[
\frac{1}{(n!)^2}  \EE_{B^1,B^2} \left[ e^{\gamma_\eps(t,B^1)+\gamma_\eps(t,B^2)} \left( \int_{[0,t]^2} R_\eps(B^1_s-B^2_u)dsdu\right)^n \right].% \leq \frac{(Ct)^n}{n!}
\]
Now we use Cauchy-Schwarz inequality and \eqref{e.a1}-\eqref{e.a2} to derive
\[
\begin{aligned}
&\EE_{B^1,B^2} \left[ e^{\gamma_\eps(t,B^1)+\gamma_\eps(t,B^2)} \left( \int_{[0,t]^2} R_\eps(B^1_s-B^2_u)dsdu\right)^n \right] \\
&\leq \sqrt{\EE_{B^1,B^2}\left[ e^{2\gamma_\eps(t,B^1)+2\gamma_\eps(t,B^2)}\right]\EE_{B^1,B^2}\left[\left( \int_{[0,t]^2} R_\eps(B^1_s-B^2_u)dsdu\right)^{2n}\right]} \\
&\les (Ct)^{n}\sqrt{(2n)!}\,.
\end{aligned}
\]
An application of Stirling's approximation yields the desired result. By Lemma~\ref{l.con2}, the same estimate holds for $f_n$. The proof is complete.

\medskip

\noindent\emph{Proof of Lemma~\ref{l.con1}}. By the same discussion as in the proof of Lemma~\ref{l.bound}, we have 
\[
\begin{aligned}
&\|f_{\epsilon, n}(\cdot; t,x)-\tilde{f}_{\eps,n}(\cdot; t,x)\|^2_{L^2(\RR^2)^{\otimes n}} \\
&\leq \frac{1}{(n!)^2}\EE_{B^1,B^2} \left[\left|(e^{\gamma_\eps(t,B^1)}-e^{\gamma(t,B^1)})(e^{\gamma_\eps(t,B^2)}-e^{\gamma(t,B^2)})\right| \left(\int_0^t\int_0^t R_\eps(B^1_s-B^2_u)dsdu\right)^n\right].
\end{aligned}
\]
By the fact that $\gamma_\eps\to \gamma$ a.s. as $\eps\to0$ and \eqref{e.a3}, we know that the random variable inside the above expectation converges to zero in probability. The uniform integrability is guaranteed by \eqref{e.a1} and \eqref{e.a2}. Thus, the r.h.s. of the above display goes to zero as $\eps\to0$.
\medskip

\noindent\emph{Proof of Lemma~\ref{l.con2}}. First, we claim that $\tilde{f}_{\eps,n}(\cdot;t,x)$ is a Cauchy sequence in $L^2(\RR^2)^{\otimes n}$. It suffices to prove the convergence of 
\begin{equation}\label{e.coneps12}
%\int_{\R^{2n}}\tilde{f}_{\eps_1,n}(y_1,\ldots,y_n;t,x)\tilde{f}_{\eps_2,n}(y_1,\ldots,y_n;t,x)dy_1\ldots dy_n
\lim_{\eps_1,\eps_2\to0}\langle  \tilde{f}_{\eps_1,n}(\cdot;t,x), \tilde{f}_{\eps_2,n}(\cdot; t,x)\rangle_{L^2(\RR^2)^{\otimes n}}.
\end{equation}
 By applying Lemma~\ref{l.localtime}, we have
\[
\EE_{B^1,B^2}\left[\prod_{j=1}^2 u_0(x+B^j_t)e^{\gamma(t,B^j)}\left(\int_0^t\int_0^t R_{\eps_1,\eps_2}(B^1_s-B^2_u)dsdu\right)^n\right]
\]
converges as $\eps_1,\eps_2\to0$, where $R_{\eps_1,\eps_2}:=\varphi_{\eps_1}\star\varphi_{\eps_2}$. This proves \eqref{e.coneps12}.

Next, we show that $\tilde{f}_{\eps,n}(\cdot;t,x)\to f_n(\cdot;t,x)$ in $L^1(\R^2)^{\otimes n}$ which implies that $f_n(\cdot;t,x) \in L^2(\R^2)^{\otimes n}$ and completes the proof. We have 
\begin{equation}\label{e.nov161}
\begin{aligned}
&\tilde{f}_{\eps,n}(y_1,\ldots,y_n; t,x)\\
&=\frac{e^{t(\mu_1+\mu_2\log t)+r_\eps}}{n!}\EE_B\left[u_0(x+B_t) e^{\gamma(t,B)} \prod_{k=1}^n \Phi^\eps_{t,x,B}(y_k)\right]\\
&=\frac{e^{t(\mu_1+\mu_2\log t)+F_t+r_\eps}}{n!}\hat{\EE}_{t,B}\left[u_0(x+B_t) \prod_{k=1}^n \Phi^\eps_{t,x,B}(y_k)\right]\\
&=e^{t(\mu_1+\mu_2\log t)+F_t+r_\eps}\varphi_\eps^{\otimes n}\star \mathscr{G}(y_1-x,\ldots,y_n-x),%\int_{\R^{2(n+1)}}\int_{[0,t]_<^n}u_0(x+z_{n+1})\prod_{k=1}^n \varphi_\eps(x+z_k-y_k)\mathscr{F}_{s_1,\ldots,s_n,t}^n(z_1,\ldots,z_n,z_{n+1})dsdz\\
%&=e^{t(\mu_1+\mu_2\log t)+F_t+r_\eps}
\end{aligned}
\end{equation}
with
\[
\mathscr{G}(z_1,\ldots,z_n):=\int_{\R^2} \int_{[0,t]_<^n}u_0(x+z_{n+1})\mathscr{F}_{s_1,\ldots,s_n,t}(z_1,\ldots,z_n,z_{n+1})dsdz_{n+1}\in L^1(\R^2)^{\otimes n}.
\]
%then the integral in \eqref{e.nov161} can be written as $\varphi_\eps^{\otimes n}\star \mathscr{G}(y_1-x,\ldots,y_n-x)$, and
Since $\varphi_\eps^{\otimes n}$ is an approximation to identity, by the classical convolution theorem,
\[
\varphi_\eps^{\otimes n}\star \mathscr{G}\to\mathscr{G} \mbox{ in } L^1(\R^2)^{\otimes n}\,.
\]
Thus,
\[
\begin{aligned}
\tilde{f}_{\eps,n}(y_1,\ldots,y_n; t,x)\to& e^{t(\mu_1+\mu_2\log t)+F_t} \mathscr{G}(y_1-x,\ldots,y_n-x)\\
&=f_n(y_1,\ldots,y_n; t,x)
\end{aligned}
\]
in $L^1(\R^2)^{\otimes n}$. %This also shows that the kernels $f_n(\cdot; t,x)$ in Theorem \ref{t.mainth} are $L^2(\R^2)^{\otimes n}$ valued functions. 

\appendix
\section{Technical lemmas}
\subsection{Measurability of $\mathscr{F}$}

We show that $\mathscr{F}_{s_1,\ldots,s_n}(x_1,\ldots,x_n)$ is jointly measurable in the $(s,x)$ variable. Fix any $0<s_1<\ldots<s_n\leq t$, consider 
\[
\begin{aligned}
\mathscr{F}^\eps_{s_1,\ldots,s_n}(x_1,\ldots,x_n)&:=\hat{\EE}_{t,B}\left[\prod_{j=1}^n\varphi_\eps(B_{s_j}-x_j)\right]\\
&=\int_{\R^{2n}}\prod_{j=1}^n \varphi_\eps(y_j-x_j)\mathscr{F}_{s_1,\ldots,s_n}(y_1,\ldots,y_n)dy.
\end{aligned}
\]
The last integral converges in $L^1(\R^{2n})$ to $\mathscr{F}_{s_1,\ldots,s_n}$. It is clear that $\mathscr{F}^\eps_{s_1,\ldots,s_n}(x_1,\ldots,x_n)$ is continuous in both $s$ and $x$ variable, hence it is measurable. If we can show $\mathscr{F}^\eps_{s_1,\ldots,s_n}(x_1,\ldots,x_n)$ converges in $L^1([0,t]_<^n\times \R^{2n})$ to some $g_{s_1,\ldots,s_n}(x_1,\ldots,x_n)$, then $g=\mathscr{F}$ almost everywhere in $[0,t]_<^n\times \R^{2n}$, which implies $\mathscr{F}$ is measurable.

 For fixed $s_1,\ldots,s_n$, we have 
\[
\int_{\R^{2n}} |\mathscr{F}^\eps_{s_1,\ldots,s_n}(x_1,\ldots,x_n)-\mathscr{F}^\delta_{s_1,\ldots,s_n}(x_1,\ldots,x_n)|dx\to0
\]
as $\eps,\delta\to0$. In addition,
\[
\begin{aligned}
&\int_{\R^{2n}} |\mathscr{F}^\eps_{s_1,\ldots,s_n}(x_1,\ldots,x_n)-\mathscr{F}^\delta_{s_1,\ldots,s_n}(x_1,\ldots,x_n)|dx\\
&\leq \int_{\R^{2n}} \left(\mathscr{F}^\eps_{s_1,\ldots,s_n}(x_1,\ldots,x_n)+\mathscr{F}^\delta_{s_1,\ldots,s_n}(x_1,\ldots,x_n)\right)dx =2.
\end{aligned}
\]
Thus, by the dominated convergence theorem, we have 
\[
\int_{[0,t]_<^n} \int_{\R^{2n}} |\mathscr{F}^\eps_{s_1,\ldots,s_n}(x_1,\ldots,x_n)-\mathscr{F}^\delta_{s_1,\ldots,s_n}(x_1,\ldots,x_n)|dxds\to0
\]
as $\eps,\delta\to0$. This completes the proof.

\subsection{Estimates on intersection local time}
We collect some standard estimates on the intersection local time of planar Brownian motion. Recall that $R_{\eps_1,\eps_2}=\varphi_{\eps_1}\star \varphi_{\eps_2}$, and assume that the Brownian motion is built on the probability space $(\Sigma, \mathcal{A},\mathbb{P}_B)$.
\begin{lemma}\label{l.localtime}
For any $\lambda>0$, there exist constants $C,t_0>0$ such that 
\begin{equation}\label{e.a1}
\sup_{\eps\in(0,1],t \in[0,t_0]}\EE_B[e^{\lambda \gamma_\eps(t,B)}]\leq C\,,
\end{equation}
and for all $n\in \mathbb{N}$,
\begin{equation}\label{e.a2}
\sup_{\eps_1,\eps_2\in(0,1]}\EE_{B^1,B^2}\left[\left(\int_{[0,t]^2} R_{\eps_1,\eps_2}(B^1_s-B^2_u)dsdu\right)^n\right] \leq n!(Ct)^n\,.
\end{equation}
In addition, 
\begin{equation}\label{e.a3}
\int_{[0,t]^2} R_{\eps_1,\eps_2}(B^1_s-B^2_u)dsdu\to \int_{[0,t]^2} \delta(B^1_s-B^2_u)dsdu
\end{equation}
in $L^2(\Sigma)$ as $\eps_1,\eps_2\to0$, where the r.h.s. is the so-called mutual intersection local time of planar Brownian motions.
\end{lemma}

\begin{proof}
The uniform exponential integrability \eqref{e.a1} is shown in \cite[Lemma A.1]{GX}.  It also contains a moment estimate of the form
\[
\begin{aligned}
&\sup_{\eps\in(0,1]}\EE_{B^1,B^2}\left[\left(\int_{[0,t]^2} R_{\eps}(B^1_s-B^2_u)dsdu\right)^n\right]\\
& \leq \EE_{B^1,B^2}\left[\left(\int_{[0,t]^2} \delta(B^1_s-B^2_u)dsdu\right)^n\right]  \leq n!(Ct)^n.
\end{aligned}
\]
The same proof leads to \eqref{e.a2}. 

Since 
\[
\int_{[0,t]^2}R_\eps(B^1_s-B^2_u)dsdu\to \int_{[0,t]^2} \delta(B^1_s-B^2_u)dsdu
\]
in $L^2(\Sigma)$, to prove \eqref{e.a3}, it suffices to show that as $\eps_1,\eps_2\to0$,
\[
\int_{[0,t]^2} R_{\eps_1,\eps_2}(B^1_s-B^2_u)dsdu-\int_{[0,t]^2} R_{\eps_1}(B^1_s-B^2_u)dsdu\to0
\]
in $L^2(\Sigma)$, which reduces to the convergence of 
\[
\EE_{B^1,B^2}\left[\int_{[0,t]^4}R_{\eps_1,\eps_2}(B^1_{s_1}-B^2_{u_1}) R_{\eps_3,\eps_4}(B^1_{s_2}-B^2_{u_2})dsdu\right]
\]
as $\eps_j\to0, j=1,2,3,4$. %By writing $R(x)=(2\pi)^{-2}\int_{\R^2}\hat{R}(\xi)e^{i\xi\cdot x}d\xi$, we only need to note that 
We write $R_{\eps_i,\eps_j}$ in the Fourier domain so that the above expectation equals to 
\[
\frac{1}{(2\pi)^4}\int_{[0,t]^4}\int_{\R^4} \hat{\varphi}(\eps_1\xi)\hat{\varphi}(\eps_2\xi) \hat{\varphi}(\eps_3\eta)\hat{\varphi}(\eps_4\eta)\EE_{B^1,B^2}[e^{i\xi\cdot (B^1_{s_1}-B^2_{u_1})} e^{i\eta\cdot (B^1_{s_2}-B^2_{u_2})}]d\xi d\eta dsdu \, .
\]
%\[
%R_{\eps_1,\eps_2}(x)=\frac{1}{(2\pi)^2}\int_{\R^2} \hat{\varphi}(\eps_1\xi)\hat{\varphi}(\eps_2\xi) e^{i\xi\cdot x}d\xi,
%\]
%and 
It suffices to use the bound
\[
\int_{[0,t]^4}\int_{\R^{4}}\EE_{B^1,B^2}[e^{i\xi\cdot (B^1_{s_1}-B^2_{u_1})} e^{i\eta\cdot (B^1_{s_2}-B^2_{u_2})}]d\xi d\eta dsdu<\infty
\]
and the dominated convergence theorem to complete the proof.
\end{proof}

\subsection*{Acknowledgments}   We would like to thank the anonymous referee for a very careful reading of the manuscript and many useful suggestions that helped improve the presentation. YG is partially supported by the NSF through DMS-1613301.


\begin{thebibliography}{99}

%\bibitem{BCR}Richard Bass, Xia Chen and Jay Rosen: Large deviations for renormalized self-intersection local times of stable processes. {\it Ann. Probab}. Volume 33, Number 3 (2005), 984--1013.


\bibitem{allez2015continuous}
{\sc R.~Allez and K.~Chouk}, {\em The continuous {A}nderson hamiltonian in
  dimension two}, arXiv preprint arXiv:1511.02718,  (2015).

\bibitem{cannizzaro2015multidimensional}
{\sc G.~Cannizzaro and K.~Chouk}, {\em Multidimensional {SDE}s with singular
  drift and universal construction of the polymer measure with white noise
  potential}, to appear in Annals of Probability, (2015).
  
  \bibitem{caravenna2013polynomial}
{\sc F.~Caravenna, R.~Sun, and N.~Zygouras}, {\em Polynomial chaos and scaling limits of disordered systems},
J. Eur. Math. Soc. 19 (2017), 1-65.

\bibitem{chen2010random}
{\sc X.~Chen}, {\em Random walk intersections: Large deviations and related
  topics}, no.~157, American Mathematical Soc., 2010.
  
  \bibitem{chouk2016invariance}
{\sc K.~Chouk, J.~Gairing, and N.~Perkowski}, {\em An invariance principle for
  the two-dimensional parabolic {A}nderson model with small potential}, Stoch PDE: Anal Comp (2017) 5: 520. 
  
    \bibitem{GX}
{\sc Y.~Gu and W.~Xu}, {\em Moments of 2D parabolic Anderson model}, to appear in Asymptotic Analysis, (2017).

  
\bibitem{gubinelli2015paracontrolled}
{\sc M.~Gubinelli, P.~Imkeller, and N.~Perkowski}, {\em Paracontrolled
  distributions and singular PDEs}, in Forum of Mathematics, Pi, vol.~3,
  Cambridge Univ Press, 2015, p.~e6.
  
 
\bibitem{hairer2014theory}
{\sc M.~Hairer}, {\em A theory of regularity structures}, Inventiones
  mathematicae, 198 (2014), pp.~269--504.

%\bibitem{hairer2015multiplicative}
%{\sc M.~Hairer and C.~Labb{\'e}}, {\em Multiplicative stochastic heat equations
%  on the whole space}, arXiv preprint arXiv:1504.07162,  (2015).

\bibitem{hairer2015simple}
{\sc M.~Hairer and C.~Labb{\'e}}, {\em A simple
  construction of the continuum parabolic Anderson model on $\mathbf{R}^{2}$},
  Electronic Communications in Probability, 20 (2015).
  
  \bibitem{hu2002chaos}
{\sc Y.~Hu}, {\em Chaos expansion of heat equations with white noise
  potentials}, Potential Analysis, 16 (2002), pp.~45--66.

\bibitem{konig}
{\sc W.~K\"{o}nig}, {\em The Parabolic Anderson Model: Random Walk in Random Potential}, Birkh\"{a}user Applied Probability and Statistics, 2016. 


\bibitem{le1992some}
{\sc J.-F. Le~Gall}, {\em Some properties of planar Brownian motion}, in Ecole
  d'Et{\'e} de Probabilit{\'e}s de Saint-Flour XX-1990, Springer, 1992,
  pp.~111--229.

\bibitem{le1994exponential}
\leavevmode\vrule height 2pt depth -1.6pt width 23pt, {\em Exponential moments
  for the renormalized self-intersection local time of planar Brownian motion},
  in S{\'e}minaire de Probabilit{\'e}s XXVIII, Springer, 1994, pp.~172--180.
  
  \bibitem{MP}
  {\sc J.~Martin, and N.~Perkowski}, {\em Paracontrolled distributions on Bravais lattices and weak universality of the 2d parabolic Anderson model}, arXiv preprint arXiv:1704.08653, (2017).
  
  \bibitem{Nualart} 
{\sc D.~Nualart}, {\em The Malliavin Calculus and Related Topics}. {Springer}, 2006. 

%\bibitem{Rosen} Jay Rosen: Joint continuity of renormalized intersection local times. {\it Annales de l'I.H.P. Probabilités et statistiques}, Volume 32 (1996) no. 6 , p. 671--700
%

\bibitem{Stroock} 
{\sc D.~Stroock}, {\em Homogeneous chaos revisited}. In: {\it Seminaire de Probabilit\'es}
XXI, Lecture Notes in Math. 1247 (1987) 1–8.

\bibitem{varadhan1969appendix}
{\sc S.~Varadhan}, {\em Appendix to Euclidean quantum field theory by K.
  Symanzik}, Local Quantum Theory. Academic Press, Reading, MA, 1 (1969).

\bibitem{yor1986precisions}
{\sc M.~Yor}, {\em Precisions sur l'existence et la continuite des temps locaux
  d'intersection du mouvement Brownien dans $\mathbf{R}^{2}$}, in S{\'e}minaire
  de Probabilit{\'e}s XX 1984/85, Springer, 1986, pp.~532--542.




\end{thebibliography}
\end{document}